\documentclass[10pt]{amsart}
\input{epsf.tex}
\usepackage{latexsym,amsfonts,amssymb,epsfig,verbatim}
\usepackage{amsmath, latexsym, graphics, textcomp}

\newcommand{\nc}{\newcommand}
\nc{\dmo}{\DeclareMathOperator}
\nc{\nt}{\newtheorem}

\theoremstyle{plain}
\nt{theorem}{Theorem}[section]
\nt*{main}{Main Theorem}
\nt{lemma}[theorem]{Lemma}
\nt{conjecture}[theorem]{Conjecture}
\nt{fact}{Fact}[section]
\nt{q}{Question}

\nt{remark}[theorem]{Remark}

\dmo{\Mod}{Mod}
\dmo{\PMod}{PMod}
\dmo{\Homeo}{Homeo}

\nc{\Z}{\mathbb{Z}}
\nc{\R}{\mathbb{R}}

\nc{\C}{\mathcal{C}}
\nc{\I}{\mathcal{I}}

\nc{\margin}[1]{\marginpar{\tiny  #1 }}

\begin{document}

\title{Two-generator subgroups of the pure braid group}

\author{Christopher J Leininger}
\author{Dan Margalit}

\address{Christopher Leininger \\ Department of Mathematics \\ University of Illinois at Urbana-Champaign \\ 1409 W. Green St. \\ Urbana, IL 61801}

\email{clein@math.uiuc.edu}

\address{Dan Margalit \\ Department of Mathematics \\ 503 Boston Ave \\ Tufts University \\ Medford, MA 02155}
\email{dan.margalit@tufts.edu}

\thanks{Both authors were supported in part by the NSF}



\maketitle

\begin{center}\today\end{center}

\begin{abstract}
We show that any two elements of the pure braid group either commute or generate a free group, settling a question of Luis Paris.  Our proof involves the theory of 3-manifolds and the theory of group actions on trees.
\end{abstract}

 \section{Introduction} \label{introsect}

Let $PB_n$ denote the pure braid group on $n$ strands; see \cite{artin} for the definition.  During a conference at Banff in 2004, Luis
Paris asked the second author whether or not two noncommuting elements of $PB_n$ could satisfy a nontrivial relation.
Our Main Theorem answers this question.

\begin{main}
\label{mainthrm}
Let $n \geq 0$, and let $g,h \in PB_n$.  Either $g$ and $h$ commute, or they generate a free group.
\end{main}

Let $F_k$ denote the free group of rank $k$, and let $\langle g,h\rangle$ denote the group generated by elements $g$
and $h$.  Since $PB_n$ is torsion free, our Main Theorem has a sharper formulation, namely, the group $\langle g,h
\rangle$ is isomorphic to the trivial group, $\Z$, $\Z^2$, or $F_2$.

As $PB_0 \cong PB_1 \cong 1$ and $PB_2 \cong \Z$, the Main Theorem is vacuous in these cases.  Also, it follows from
Artin's presentation for $PB_n$ that $PB_3 \cong F_2 \times \Z$ \cite[p. 120]{artin}, and so the Main Theorem is easy
to prove in that case.  For $n \geq 4$, the only case where the Main Theorem was already known was the case where $g$ and $h$ are both positive multitwists \cite[Theorem 3.2]{hht} \cite{ishida}.

Our Main Theorem is very specific to the pure braid group.  That is, for $n \geq 3$, the subgroup of $B_n$ generated by $PB_n$ and a single half-twist $g$ is an index 2 supergroup of $PB_n$, and does not satisfy the conclusion of the Main Theorem; indeed, there is a Dehn twist $h$ so that $g$ does not commute with $h$, but $ghgh=hghg$ (see \cite[Lemma 2.3]{parislab}).

In the early 1980's, Ivanov and McCarthy independently proved that, given two elements $g$ and $h$ of the braid group
(more generally, the mapping class group of a surface), there is an $N>0$ so that $g^N$ and $h^N$ either commute
or generate a free group \cite{ivanovtits,mctits}.  Hamidi-Tehrani gave an algorithm, when $g$ and $h$ are pseudo-Anosov, for finding such a power $N$ \cite[Theorem 3.19]{hhtthesis}.  Fujiwara also provided criteria for finding an appropriate $N$ \cite[Proposition 1.1 and Theorem 3.1]{fujiwara}, and has announced that one can find such an $N$ that depends only on the surface \cite{fujiwara2}. Our Main Theorem can be rephrased as saying that for $PB_n$, we can take $N=1$.  Also, as an immediate consequence, for the braid group $B_n$, one can take $N=n!$.

Appealing to Fujiwara's work \cite{fujiwara}, Mangahas proved that any finitely generated subgroup of a mapping class
group that is not virtually abelian has exponential growth, with the growth rate bounded away from $1$ by a constant
depending only on the surface, and not on the generating set \cite[Theorem 1.1]{mangahas}.  An immediate consequence of
our Main Theorem is that any nonabelian subgroup $G < PB_n$ has growth rate at least $3$, independently of the
generating set, and independently of $n$ (compare with the linear setting \cite{breuillardgelander}).

\noindent {\bf Acknowledgments.}  We would like to thank Javier Aramayona, Mladen Bestvina, Tom Church, Nathan
Dunfield, Daniel Groves, Stephen Humphries, Ilya Kapovich, Lars Louder, and Luis Paris for helpful conversations and
suggestions.  We would also like to thank Benson Farb for hosting a conference on Torelli groups in 2006, where this
project was begun. Finally, we would especially like to thank Nathan Broaddus for discussing the problem and for
writing a computer program that searched for relations in the pure braid group.

\section{Reduction to mapping class groups} \label{reduction}

The center of $PB_n$ is isomorphic to $\Z$, and the quotient of $PB_n$ by its center is naturally isomorphic to
$\PMod(S_{0,n+1})$, the \textit{pure mapping class group} of the $(n+1)$-times punctured sphere.  This is the subgroup
of the mapping class group $\Mod(S_{0,n+1})$ consisting of those mapping classes that fix each puncture.  Moreover,
$PB_n$ splits over its center:
\[ PB_n \cong \PMod(S_{0,n+1}) \times \Z \]
(see \cite[\S 8.2]{primer}).

The Main Theorem follows easily from the next theorem, whose proof will occupy most of the paper.
\begin{theorem} \label{modmain}
Let $n \geq 0$ and let $g,h \in \PMod(S_{0,n})$.  The group $\langle g,h \rangle$ is either abelian or is isomorphic to $F_2$.
\end{theorem}

\begin{proof}[Proof of the Main Theorem assuming Theorem~\ref{modmain}.]

Let $G$ denote $\langle g, h \rangle$, and let $\bar G$ denote the image of $G$ under the map $PB_n \to PB_n/\Z \cong \PMod(S_{0,n+1})$.  By Theorem \ref{modmain}, $\bar G$ is either abelian or is isomorphic to $F_2$.  If $\bar G$ is abelian, then, as $G$ is a split central extension of $\bar G$ by an abelian group, it follows that $G$ is abelian.  If $\bar G \cong F_2$, then we have a sequence of maps $F_2 \to G \to \bar G \cong F_2$ where the first map sends the generators of $F_2$ to $g$ and $h$.  Since free groups are Hopfian \cite[Theorem 2.13]{magkarsol}, and since both maps are surjective, it follows that both maps are isomorphisms, and so $G \cong F_2$.
\end{proof}

\section{Filling a puncture} \label{fillerup}

In Section~\ref{section:proof}, we will prove Theorem~\ref{modmain} by means of induction.  The inductive step is
realized by ``filling in'' one of the punctures of $S_{0,n}$.  In this section, we recall some group theoretic,
combinatorial, and geometric consequences of the filling operation.

\subsection{The Birman exact sequence} \label{bes}

The surface $S_{0,n}$ is obtained from the 2-sphere by removing a set of $n$ points, say $p_1, \dots, p_n$. Filling the
$k^{th}$ puncture back in, we obtain $S_{0,n-1}$, so that $S_{0,n} = S_{0,n-1} - \{p_k\}$.  Via the inclusion $i_k :
S_{0,n} \to S_{0,n-1}$, any homeomorphism of $S_{0,n}$ that induces the trivial permutation of the punctures extends
uniquely to a homeomorphism of $S_{0,n-1}$, and so induces a surjective homomorphism $\rho_k:\PMod(S_{0,n}) \to
\PMod(S_{0,n-1})$.  The Birman exact sequence \cite[\S1]{birman} describes the kernel of this map for $n \geq 4$:
\[ 1 \to \pi_1(S_{0,n-1},p_k) \to \PMod(S_{0,n}) \stackrel{\rho_k}{\to} \PMod(S_{0,n-1}) \to 1.\]

\subsection{Curve complexes and trees} \label{cctrees}

The complex of curves of a surface $S$ is the simplicial complex whose vertices are isotopy classes of essential simple
closed curves in $S$, and whose $k$-simplices are sets of $k+1$ isotopy classes of pairwise disjoint curves. We denote
this complex $\C(S)$.  Here, a curve is \emph{essential} if it is neither homotopic to a point nor a puncture.

One sometimes modifies the definition of the curve complex for $S_{0,4}$ so as to obtain the Farey graph.
We do not do this, and hence $\C(S_{0,4})$ is a countable set of vertices.

The inclusion $i_k: S_{0,n} \to S_{0,n-1}$ induces a map $\Pi_k$ from the set of isotopy classes of curves in $S_{0,n}$
to the set of isotopy classes of curves in $S_{0,n-1}$. There is a well defined subcomplex $\Pi_k^{-1}(v)$ of
$\C(S_{0,n})$, which consists of the union of all simplices of $u \subset \C(S_{0,n})$, such that upon filling in the
$k^{th}$ puncture, all vertices of $u$ represent $v$ in $S_{0,n-1}$. Via the inclusion $\pi_1(S_{0,n-1}) \to
\PMod(S_{0,n})$ from the Birman exact sequence, we obtain an action of $\pi_1(S_{0,n-1})$ on $\Pi_k^{-1}(v)$. With Kent
and Schleimer, the first author showed that $\Pi_k^{-1}(v)$ is $\pi_1(S_{0,n-1})$--equivariantly isomorphic to the
Bass--Serre tree $T_v$ associated to the splitting of $\pi_1(S_{0,n-1})$ defined by $v$ \cite[Theorem 7.1]{KLS}.

We note that if $G$ is a subgroup of $\PMod(S_{0,n-1})$ that fixes a vertex $v$ of $\C(S_{0,n-1})$, then $\rho_k^{-1}(G)$ preserves $\Pi_k^{-1}(v)$, that is, $\rho_k^{-1}(G)$ acts on the tree $\Pi_k^{-1}(v) \cong T_v$.

\subsection{Brunnian subgroups} \label{brunniansect}

Whittlesey considered the \emph{Brunnian subgroup} of the group $\PMod(S_{0,n})$, which is defined as
\[ W_n = \bigcap_{k=1}^n \ker(\rho_k) < \PMod(S_{0,n}), \]
and she proved that for $n \geq 5$ each nontrivial element of $W_n$ is irreducible, which means that it does not
fix any simplex of $\C(S_{0,n})$ \cite[Theorem 4.1]{whittle}.  Since the group $\PMod(S_{0,n})$ is torsion free, the
Nielsen--Thurston classification gives that irreducible elements of $\PMod(S_{0,n})$---in particular, elements of
$W_n$---are pseudo-Anosov \cite[\S 5]{thurston}

\subsection{Surface bundles} \label{bundles}

Any nontrivial element $f$ of  $\PMod(S_{0,n-1})$ generates an infinite cyclic group $\langle f \rangle <
\PMod(S_{0,n-1})$.   The mapping torus $M_f$ of $f$ is an $(S_{0,n-1})$--bundle over the circle, and
\[\pi_1(M_f) \cong \rho_k^{-1}(\langle f \rangle).\]
Since the universal cover of $M_f$ is homeomorphic to $\R^3$, it follows that $M_f$ is irreducible, that is, every 2-sphere bounds a 3-ball.

If $f$ is pseudo-Anosov, then any conjugacy class of $\pi_1(S_{0,n-1})$ fixed by a nontrivial power of $f$ is peripheral.  It follows that $M_f$ is atoroidal, that is, any immersed, $\pi_1$--injective torus in $M_f$ is peripheral.

When $f$ is pseudo-Anosov, the fact that $M_f$ is irreducible and atoroidal also follows immediately from the theorem of Thurston that states that $M_f$ admits a complete finite-volume hyperbolic metric; see \cite[Theorem 0.1]{thurstonhyp2} and \cite{otalfiber}.

\section{Proof of Theorem~\ref{modmain}}
\label{section:proof}

Since $\PMod(S_{0,n})$ is trivial for $n \leq 3$, and $\PMod(S_{0,4}) \cong F_2$ by the Birman exact sequence, there is
nothing to do for these cases.  We proceed by induction. Say that $n \geq 5$, and assume that Theorem~\ref{modmain}
holds for $\PMod(S_{0,n-1})$.  In other words, assume that if $\bar g$ and $\bar h$ are any two elements of
$\PMod(S_{0,n-1})$, then $\langle \bar g, \bar h \rangle$ is either abelian or is isomorphic to $F_2$.

Let $g$ and $h$ be two fixed elements of $\PMod(S_{0,n})$.  As in Section~\ref{bes}, let $\rho_1, \dots, \rho_n$ be the maps
$\PMod(S_{0,n}) \to \PMod(S_{0,n-1})$ obtained by filling a puncture.  We consider three possibilities for the
subgroups $\bar G_k = \langle \rho_k(g), \rho_k(h) \rangle$.  The division into cases arises as follows: either $\bar
G_k$ is isomorphic to $F_2$ for some $k$ (Case 1 below), or $\bar G_k$ is abelian for all $k$.  In the latter situation, we
consider any particular $\bar G_k$ and appeal to following dichotomy of abelian subgroups of $\PMod(S_{0,n})$, due to Ivanov \cite[Corollary 1.8 plus Corollary 8.5]{ivanov}.
\begin{quote}
\emph{Any abelian subgroup of $\PMod(S_{0,n})$ either fixes a vertex of $\C(S_{0,n})$ or
is an infinite cyclic group generated by a pseudo-Anosov.}
\end{quote}

The two possibilities for abelian subgroups are dealt with in Case 2 and Case 3.

\noindent \textbf{Case 1}. $\bar G_k \cong F_2$ for some $k$.\\

\noindent By the Hopfian property for free groups, it follows that $\langle g,h \rangle \cong F_2$.\\

\noindent \textbf{Case 2.} $\bar G_k \cong \Z$ and is generated by a pseudo-Anosov mapping class for some $k$.\\

\noindent Let $\langle f \rangle = \langle \rho_k(g),\rho_k(h) \rangle$ with $f$ pseudo-Anosov.  In this case, we know
that $\rho_k^{-1}(\langle f \rangle)$ is isomorphic to the fundamental group of the mapping torus $M_f$ (see
Section~\ref{bundles}).  Since $M_f$ is atoroidal and irreducible, a theorem of Jaco--Shalen gives that $\langle g, h \rangle$ either has finite index in $\rho_k^{-1}(\langle f
\rangle) \cong \pi_1(M_f)$, or it is isomorphic to one of the groups $1$, $\Z$, $\Z^2$, or $F_2$; see \cite[Theorem
VI.4.1]{jacoshalen}.  Thus, it suffices to show that $\langle g,h \rangle$ does not have finite index in $\pi_1(M_f)$.

Since $f \in \PMod(S_{0,n-1})$ acts trivially on $H_1(S_{0,n-1},\Z) \cong \Z^{n-2}$, a Mayer--Vietoris computation
shows that the first homology of $M_f$ is isomorphic to that of $S_{0,n-1} \times S^1$, that is
\[ H_1(M_f,\Z) \cong \Z^{n-1}.\]
Given a finite sheeted covering $M \to M_f$, we see that $H_1(M,\Z) \cong \pi_1(M)^{ab}$ contains a free abelian
group of rank $n-1 \geq 4$.  Therefore, any finite index subgroup of $\pi_1(M)$ is generated by no fewer than $4$
elements. It follows that $\langle g,h \rangle
< \rho_k^{-1}(\langle f \rangle) \cong \pi_1(M_f)$ does not have finite index, which completes the proof in Case 2.\\

\noindent \textbf{Case 3.} $\bar G_k$ is abelian for all $k$, and fixes a vertex $v \in \C(S_{0,n-1})$ for some $k$.\\

\noindent Set $G = \langle g,h \rangle$, and note that $[G,G]$ is a subgroup of $W_n$, the Brunnian subgroup of
$\PMod(S_{0,n})$.  This is because $\rho_k$ factors through the abelianization of $G$ for all $k$, and so $[G,G]$ is
contained in the intersection of the kernels of all $\rho_k$.

Choose $k$ so that $\bar G_k$ fixes a vertex $v \in \C(S_{0,n-1})$.  Then $G < \rho_k^{-1}(\bar G_k)$, and, as described
in Section \ref{cctrees}, $\rho_k^{-1}(\bar G_k)$ acts on the tree $\Pi_k^{-1}(v) \cong T_v$.  Note that
$\rho_k^{-1}(\bar G_k)$ contains $\pi_1(S_{0,n-1})$, and the action of $\pi_1(S_{0,n-1})$ on $\Pi^{-1}(v)$ is
equivalent to the action of $\pi_1(S_{0,n-1})$ on its Bass--Serre tree $T_v$.  Therefore, all edge stabilizers are nontrivial.  Moreover, $\rho_k^{-1}(\bar G_k)$ acts on $\Pi_k^{-1}(v)$ without inversions; indeed, any element of $\Mod(S_{0,n})$ that interchanges two disjoint isotopy classes of curves is easily seen to act nontrivially on the punctures of $S_{0,n}$.

According to the work of Kapovich--Weidmann \cite[Theorem 1]{kapovichweidmann}, either $G$ is abelian or free, or one
can find two new generators $s,t$ for $G$ so that one of the following holds:
\begin{enumerate}
\item $s$ and $t$ have nonzero powers that fix a common point of $\Pi_k^{-1}(v)$, or
\item $s$ and $tst^{-1}$ have nonzero powers that fix a common point of $\Pi_k^{-1}(v)$.
\end{enumerate}
We can strengthen the conclusion of the Kapovich--Wiedmann theorem by applying the following fact, due to Ivanov \cite[Corollary 1.8 plus Corollary 3.7]{ivanov}.
\begin{quote}
\emph{If $f \in \PMod(S_{0,n})$ has $f^n(v) = v$ for some vertex $v \in \C(S_{0,n})$ and some nonzero $n$, then $f(v) = v$.}
\end{quote}
We can thus conclude that either $s$ and $t$, or $s$ and $tst^{-1}$, have a common fixed point in $\Pi_k^{-1}(v)$.  Without loss of generality, let this fixed point be a vertex $u$.  It follows that the commutator $[s,t]$ also fixes $u$. However, $[s,t] \in [G,G] < W_n$, and hence $[s,t] = 1$ (see Section \ref{brunniansect}).  Therefore, $G$ is abelian, completing the proof in Case 3, and thus finishing the proof of Theorem~\ref{modmain}.

\section{Questions and Conjectures}

\paragraph{\bf Mapping class groups} As we alluded to in the introduction, one interpretation of our Main Theorem is that if $g$ and $h$ are any two elements of $\Mod(S_{0,n})$, then $\langle f^{n!}, g^{n!} \rangle$ is either abelian or free.

\begin{q}
Let $S$ be a surface of positive genus.  Is there a constant $N=N(S)>0$ so that, if $f,g \in \Mod(S)$, then $\langle f^N, g^N \rangle$ is either abelian or free?
\end{q}

It follows from \cite[Theorem 3.2]{hht} and \cite{ishida} that we can take $N = 2$ if $g$ and $h$ are Dehn twists or positive multitwists.  As mentioned above, Fujiwara provides an appropriate $N = N(S)$, under the additional assumption that $f$ and $g$ are pseudo-Anosov \cite{fujiwara2}.

\paragraph{\bf Torelli groups and congruence subgroups} Let $S$ be a closed surface.  The \emph{Torelli group} $\I(S)$ is the subgroup of $\Mod(S)$ consisting of those elements that act trivially on $H_1(S,\Z)$.

\begin{conjecture}
\label{conj:torelli} For any closed surface $S$, and any $g,h \in \I(S)$, either $g$ and $h$ commute, or
they generate a free group.
\end{conjecture}

In the case where $g$ and $h$ are Dehn twists or positive multitwists in either $PB_n$ or $\I(S)$, both the Main Theorem and
Conjecture~\ref{conj:torelli} have been known for some time; see \cite[Theorem 3.2]{hht} and \cite{ishida}.  On the other hand, it is unknown whether or not two non-commuting bounding pair maps in $\I(S)$ necessarily generate a free group.

Our Main Theorem would follow immediately from Conjecture~\ref{conj:torelli}, since, if the genus of $S$ is large enough, $\I(S)$ contains a split central extension of $PB_n$.

Conjecture~\ref{conj:torelli} is not true if we replace $\I(S)$ with the level $m$ congruence subgroup of the mapping class group, that is, the subgroup consisting of elements that act trivially on $H_1(S,\Z/m\Z)$.  Indeed, if $f$ is any bounding pair map, then there is an isotopy class of simple closed curves $a$ where $f(a) \neq a$ and $f(a)$ is disjoint from $a$.  Thus, $f$ and $T_a^m$ fail to commute, but $T_a^m$ commutes with the element $fT_a^mf^{-1}=T_{f(a)}^m$, and so there is a nontrivial relation in $\langle f, T_a^m \rangle$, namely, $[T_a^m,fT_a^mf^{-1}]=1$.

\paragraph{\bf Surface braid groups} Tom Church has asked us whether the Main Theorem is true for pure surface braid groups.  That is, do any two elements in any
pure surface braid group commute or generate a free group?

We conjecture that the answer is `yes'.  After rephrasing in terms of mapping class groups via Birman's work
\cite[Theorem 1]{birman}, one can see that the majority of our proof applies.  The only missing ingredient is
Whittlesey's result---the intersection of the ``fill-in-a-puncture'' homomorphisms is not purely pseudo-Anosov on a
general surface.

\newpage

\bibliographystyle{plain}
\bibliography{pbrels}

\end{document}